\documentclass[11pt, a4paper]{article}

\usepackage{amsmath, amssymb, amsthm, xspace, a4wide, amscd}
\usepackage[english]{babel}
\usepackage[T1]{fontenc}
\usepackage[latin1]{inputenc}

\newcommand{\ff}{\mathcal{F}}
\newcommand{\bb}{\mathcal{B}}
\newcommand{\rr}{\mathbb{R}}
\newcommand{\cc}{\mathbb{C}}
\newcommand{\pp}{{\sf{P}}}
\newcommand{\xx}{{\sf X}}
\newcommand{\pr}{\partial}
\newcommand{\llf}{{\sf L}}

\newtheorem{thm}{Theorem}
\newtheorem{lemma}{Lemma}
\newtheorem{assumption}{Assumption A\hspace*{-4pt}}

\theoremstyle{definition}
\newtheorem{defin}{Definition}

\theoremstyle{remark}
\newtheorem{remark}{Remark}

\begin{document}

\begin{center}\Large
Transport equation  driven by a stochastic measure
\end{center}

\begin{center}
Vadym Radchenko \footnote{Department of Mathematical Analysis, Taras Shevchenko National University of Kyiv,
01601 Kyiv, {Ukraine},\\ vadymradchenko@knu.ua}
\footnote{This work was supported by Alexander von Humboldt Foundation, grant 1074615. The author is grateful to Prof. M.~Z\"{a}hle for fruitful discussions during the preparation of this paper and thanks the Friedrich-Schiller-University of Jena for its hospitality.}
\end{center}

\begin{abstract}
We consider the stochastic transport equation where the randomness is given by the symmetric integral with respect to stochastic measure. For stochastic measure, we assume only $\sigma$-additivity in probability and continuity of paths. The existence and uniqueness of the weak solution to the equation are proved.
\end{abstract}

\section{Introduction}\label{scintr}

We consider the stochastic transport equation that formally can be written in the form
\begin{equation}\label{eqtrans}
\begin{split}
\dfrac{\pr u(t,x)}{\pr t}\,{d}t+b(t,x)\dfrac{\pr u(t,x)}{\pr x}\,{d}t+\dfrac{\pr u(t,x)}{\pr x}\circ{d}\mu(t)=0,\\
u(0,x)=u_0(x),\quad x\in\rr,\quad t\in [0,T].
\end{split}
\end{equation}
Here $\mu$ is a stochastic measure (SM), see Definition~\ref{dfstme} below. We assume that $\mu$ is defined on the Borel $\sigma$-algebra of $[0,T]$, and the process $\mu_t=\mu((0,t])$ has a continuous paths. Assumptions on $b$ and $u_0$ are given in Section~\ref{sceqfo}. Equation~\eqref{eqtrans} we consider in the weak sense, definition of the solution is given in~\eqref{eqweaks}.

We will prove existence and uniqueness of the solution. Similarly to other types of stochastic transport equation, we demonstrate that the solution is given by formula $u(t,x)=u_0(X_t^{-1}(x))$ where $X_t(x)$ satisfies the auxiliary equation~\eqref{eqaddx}.

Stochastic integral with respect to~$\mu$ is defined as symmetric integral. This Stratonovich-type integral was studied in~\cite{rads16}, we recall it's definition and basic properties in subsection~\ref{sssyin}. SMs include many important classes of processes, but we can prove existence of the integral only for integrands of the form $f(\mu_t,t)$ where $f\in {\cc}^{1,1} ({\rr}\times [0,T] )$. Thus, we will find our solution~$u$ having this form.

For the stochastic transport equation driven by the Wiener process, the existence and uniqueness of the solution were proved under different assumptions on $b$ and $u_0$, see  \cite{catoli},
\cite{fang07},
\cite{flandoli10},
\cite{moham15},
\cite{wei22dini}.
It was shown that stochastic term in the transport equation leads to regularization of the solution, see \cite{beck19}, \cite{fedflan13}, \cite{flandoli10}. Equation in bounded domain was studied in \cite{neves21}. In these papers, the stochastic term is given by Stratonovich integral and solution is considered in the weak sense. In \cite{wei21regul} the existence and uniqueness of stochastic strong solution are obtained, the renormalized weak solution was studied in~\cite{zhang10}.

Transport equation with other stochastic integrators is less studied. The existence and uniqueness of the solution to equation driven by L\'{e}vy white noise was proved in~\cite{pros04}, to equation driven by fractional Brownian motion -- in~\cite{oltud15}. In both papers, the Malliavin calculus approach was used.

In this paper, we consider the rather general stochastic integrator. At the same time, we need some restrictive assumptions on $b$ and $u_0$, and study the case of one dimensional spatial variable.

The recent results for the equations driven by stochastic measures may be found in~\cite{bodnarchuk_2020}, \cite{manikin_2022as}, \cite{manikin_2022av}.

The rest of the paper is organized as follows. In Section~\ref{scprel} we recall the definitions and basic facts concerning stochastic measures and symmetric integral. Also we prove the analogue of the Fubini theorem for our integral that we will need below. In Section~\ref{sceqfo} we give our assumptions on the equation and formulate the main result.
Section~\ref{scexis} is devoted to the proof of the existence of the solution, and we give the explicit formula for $u$. In Section~\ref{scuniq}, under some additional assumptions, we obtain the uniqueness of the solution.

\section{Preliminaries}\label{scprel}

\subsection{Stochastic measures}\label{ssstme}

Let $\llf_0=\llf_0(\Omega, {\ff}, {\pp} )$ be the set of all real-valued
random variables defined on the complete probability space $(\Omega, {\ff}, {\pp} )$ (more precisely, the set of equivalence classes). Convergence in $\llf_0$ means the convergence in probability. Let ${\xx}$ be an arbitrary set and ${\bb}$ a $\sigma$-algebra of subsets of ${\xx}$.

\begin{defin}\label{dfstme}
A $\sigma$-additive mapping $\mu:\ {\bb}\to \llf_0$ is called {\em stochastic measure} (SM).
\end{defin}

We do not assume the moment existence or martingale properties for SM. In other words, $\mu$ is $\llf_0$--valued vector measure.

Important examples of SMs are orthogonal stochastic measures, $\alpha$-stable random measures defined on a $\sigma$-algebra for $\alpha\in (0,1)\cup(1,2]$ (see \cite[Chapter 3]{samtaq}).

Many examples of the SMs on the Borel subsets of $[0,T]$ may be given by the Wiener-type integral
\begin{equation}\label{eqmuax}
\mu(A)=\int_{[0,T]} {\mathbf 1}_A(t)\,{d}X_t.
\end{equation}

We note the following cases of processes $X_t$ in~\eqref{eqmuax} that generate SM.

\begin{enumerate}

\item\label{itmart} $X_t$~-- any square integrable continuous martingale.

\item\label{itfrbr} $X_t=W_t^H$~-- the fractional Brownian motion with Hurst index $H>1/2$, see Theorem~1.1~\cite{memiva}.

\item\label{itsfrb} $X_t=S_t^k$~-- the sub-fractional Brownian motion for $k=H-1/2,\ 1/2<H<1$, see Theorem 3.2~(ii) and Remark 3.3~c) in~\cite{tudor09}.

\item\label{itrose} $X_t=Z_H^k(t)$~-- the Hermite process, $1/2<H<1$, $k\ge 1$, see~\cite{tudor07}, \cite[Section 3.1.3]{tudor13}. $Z_H^2(t)$ is known as the Rosenblatt process, see also~\cite[Section~3]{tudor08}.

\end{enumerate}

 The detailed theory of stochastic measures is presented in~\cite{radbook}.

The results of this paper will be obtained under the following assumption on $\mu$.

\begin{assumption}\label{assborcont}
$\mu$ is an SM on Borel subsets of $[0,T]$, and the process $\mu_{t}=\mu((0,t])$ has continuous paths on $[0,T]$.
\end{assumption}

Processes $X_t$ in examples \ref{itmart}--\ref{itrose} are continuous, therefore A\ref{assborcont} holds in these cases.

\subsection{Symmetric integral}\label{sssyin}

The symmetric integral of random functions with respect to stochastic measures was considered in~\cite{rads16}. We review the basic facts and definitions concerning this integral.

\begin{defin} Let $\xi_{t}$ and $\eta_{t}$ be random processes on $[0,T]$, $0=t_{0}^n<t_{1}^n<\dots<t_{j_n}^n=T$ be a sequence of partitions such that $\max_k |t_k^n-t_{k-1}^n|\to 0$, $n\to\infty$. We define
\begin{equation}\label{eqdfis}
\int_{(0,T]}\xi_{t}\circ{d}\eta_{t}:={\rm p}\lim_{n\to\infty}\sum_{k=1}^{j_n}\frac{\xi_{t_{k-1}^n}+\xi_{t_{k}^n}}{2}\, (\eta_{t_{k}^n}- \eta_{t_{k-1}^n} )
\end{equation}
provided that this limit in probability exists for any such sequence of partitions.
\end{defin}

For Wiener process $\eta_t$ and adapted $\xi_t$ we obtain the classical Stratonovich integral. If $\eta_t$ and $\xi_t$ are H\"{o}lder continuous with exponents $\gamma_\eta$ and $\gamma_\xi$, $\gamma_\eta+\gamma_\xi>1$, then value of~\eqref{eqdfis} equals to the integral defined in~\cite{zahle98}.

The following theorem describes the class of processes for which the integral is well defined.

\begin{thm} (Theorem~4.6~\cite{rads16})\label{thintf}
Let A\ref{assborcont} holds, $f\in {\cc}^{1,1} ({\rr}\times [0,T] )$. Then integral~\eqref{eqdfis} of $f(\mu_{t}, {t})$ with respect to $\mu_{t}$ is well defined, and
\begin{equation}
\int_{(0,T]}f(\mu_{t}, {t})\circ{d}\mu_{t}
=F(\mu_{T}, {T})-\int_{(0,T]} F_2'(\mu_{t},{t})\,{d}{t},\label{eqintf}
\end{equation}
where $F(z,v)=\int_{0}^z f(y,v)\,{d}y$.
\end{thm}

Some other properties and equations with the symmetric integral are considered in~\cite{rads16}, \cite{radavsm19}, \cite{radbook}.

\subsection{Fubini theorem for symmetric integral}

We will need the following auxiliary statement.

\begin{lemma}\label{lmlmfb} Let $f:\rr\times [0,T]\times\rr\to\rr$ be measurable and has continuous derivatives $f'_y(y,t,x),\ f'_t(y,t,x)$. Assume that
$$
|f(y,t,x)|\le g(x),\quad |f'_y(y,t,x)|\le g_1(x),\quad |f'_t(y,t,x)|\le g_2(x)
$$
for some $g,\ g_1,\ g_2\in \llf^1(\rr,{d}x)$. Then
\begin{equation}\label{eqlmfb}
\int_{\rr}\int_{(0,T]} f(\mu_t,t,x)\circ{d}\mu_t\,{d}x= \int_{(0,T]}  \int_{\rr}f(\mu_t,t,x)\,{d}x\,\circ {d}\mu_t.
\end{equation}
\end{lemma}

\begin{proof}
Denote
$$
F(z,t,x)=\int_{0}^z f(y,t,x)\,{d}y,\quad \tilde{F}(z,t)=\int_{0}^z \int_{\rr}f(y,t,x)\,{d}x\,{d}y ,\quad z\in\rr.
$$
Theorem~\ref{thintf} and assumptions of the lemma and  imply that the integrals in~\eqref{eqlmfb} are well-defined. Applying~\eqref{eqintf}, we transform left-hand side and right-hand side of~\eqref{eqlmfb}
\begin{eqnarray*}
\int_{\rr} \int_{(0,T]} f(\mu_t,t,x)\circ \,{d}\mu_t\, {d}x=\int_{\rr} \Bigl(F(\mu_{T},T,x)-\int_{(0,T]} F_t'(\mu_{t},t,x)\,{d}{t}\Bigr){d}x,\\
\int_{(0,T]} \int_{\rr}f(\mu_t,t,x)\,{d}x\,\circ{d}\mu_t =\tilde{F}(\mu_T,T)-\int_{(0,T]} \tilde{F}_t'(\mu_{t},t)\,{d}{t}.
\end{eqnarray*}
The equalities
\begin{eqnarray*}
\int_{\rr}F(\mu_{T},T,x)\,{d}x=\tilde{F}(\mu_T,T)\\
\Leftrightarrow \int_{\rr}  \int_{0}^{\mu_T} f(y,T,x)\,{d}y\,{d}x=\int_{0}^{\mu_T} \int_{\rr}f(y,T,x)\,{d}x\,{d}y ,\\
\int_{\rr}\int_{(0,T]} F_t'(\mu_{t},t,x)\,{d}{t}\,{d}x =\int_{(0,T]} \tilde{F}_t'(\mu_{t},t)\,{d}{t}\\
\Leftrightarrow \int_{\rr}\int_{(0,T]}  \int_{0}^{\mu_t} f'_t(y,t,x)\,{d}y\,{d}{t}\,{d}x = \int_{(0,T]}\int_{0}^{\mu_{t}}  \int_{\rr}f'_t(y,t,x)\,{d}x\,{d}y\,{d}{t}.
\end{eqnarray*}
hold by usual Fubini's theorem.
\end{proof}

\section{The problem. Formulation of the main result.}\label{sceqfo}

We consider equation~\eqref{eqtrans} in the weak form.
This means that $u:[0,T]\times \rr\times\Omega\to\rr$ is a measurable random function such that for each $\varphi\in\cc_0^{\infty}(\rr)$ holds
\begin{equation}
\begin{split}\label{eqweaks}
\int_{\rr}u(t,x)\varphi(x)\,{d}x=\int_{\rr}u_0(x)\varphi(x)\,{d}x
+\int_0^t  \int_{\rr}u(s,x)\Bigl(b(s,x)\varphi'(x)+\dfrac{\partial b(s,x)}{\partial x}\varphi(x)\Bigr)\,{d}x\,{d}s\\
+\int_0^t \int_{\rr}u(s,x)\varphi'(x)\,{d}x\,\circ {d}\mu(s).
\end{split}
\end{equation}
By $\cc_0^{\infty}(\rr)$ we denote the class of infinitely differentiable functions with the compact support.

For our equation, we will refer to the following assumptions.

\begin{assumption}\label{assfuzero} $u_0:\rr\times\Omega\to\rr$ is measurable and has continuous derivative in~$x$.
\end{assumption}

\begin{assumption}\label{assfuzerob} $|u_0(x)|\le C(\omega)$ for some finite random constant~$C(\omega)$.
\end{assumption}

\begin{assumption}\label{assfb} $b:[0,T]\times \rr\to\rr$ is continuous,  $\dfrac{\pr b(t,x)}{\pr x}$ is continuous and bounded.
\end{assumption}

\begin{assumption}\label{assfbl} $\sup_{t\in [0,T]}\int_{|x|\ge r}\dfrac{|b(t,x)|}{1+|x|}\,{d}x\to 0,\ r\to\infty$.
\end{assumption}

Note that, by A\ref{assfb}, $b$ is globally Lipschitz continuous in $x$.

For each fixed $\omega\in\Omega$,  we consider the following auxiliary equation
\begin{eqnarray}\label{eqaddx}
X_{t}(x)=x+\int_0^t b(r,X_{r}(x))\,{d}r+\mu_t,\quad 0\le t\le T.
\end{eqnarray}

Assumption~A\ref{assfb} imply that \eqref{eqaddx} has a unique solution on~$[0,T]$ for each~$x$.

By well known result of theory of ordinary differential equations, the solution has a continuous derivative
$$
X_t'(x)=\frac{\pr}{\pr x} X_t(x).
$$

We have
\begin{equation}
\begin{split}\label{eqdifx}
X_t'(x)=1+\int_0^t \dfrac{\partial b(r,X_{r}(x))}{\partial x} X'_{r}(x)\,{d}r\\
\Rightarrow\dfrac{\pr}{\pr t}X_t'(x)=\dfrac{\partial b(t,X_{t}(x))}{\partial x} X'_{t}(x)\quad
\Rightarrow\quad X'_{t}(x)=\exp\Bigl\{\int_0^t\dfrac{\partial b(s,X_{s}(x))}{\partial x}\,{d}s\Bigr\}.
\end{split}
\end{equation}
Therefore, $X'_{t}(x)>0$, and the function $X_t^{-1}(x)$, where inverse is taken with respect to variable $x$, is well defined.

Note that $X_t$ is the sum of a differentiable function of $t$ and $\mu_t$, $X'_{t}$ is a differentiable function of $t$. This allows us to consider integrals
$$
\int_{(0,T]}g(X_{t}, X'_{t},\mu_t,t)\circ{d}\mu_{t},\quad g\in {\cc}^{1,1,1,1} ({\rr}^3\times [0,T]).
$$

The main result of the paper is the following.

\begin{thm} 1) Let Assumptions A\ref{assborcont}, A\ref{assfuzero}, A\ref{assfb} hold, $X_{t}(x)$ be the solution of~\eqref{eqaddx}. Then the random function
\begin{equation}\label{eqsolut}
u(t,x)=u_0(X_t^{-1}(x))
\end{equation}
satisfies~\eqref{eqweaks}.

2) In addition, let Assumptions A\ref{assfuzerob} and A\ref{assfbl} hold. Then solution~\eqref{eqsolut} is unique in the class of measurable random functions $u(t,x)=h(\mu_t,t,x)$, such that $h(\cdot,\cdot,x)\in {\cc}^{1,1} (\rr\times [0,T])$ for each $x\in\rr$, and $|u(t,x)|\le C(\omega)$ for some finite random constant~$C(\omega)$.
\end{thm}

\begin{remark} Note that $u(t,x)=u_0(X_t^{-1}(x))$ has a form $h(\mu_t,t,x)$ from the second part of the theorem.
This follows from Assumption~A\ref{assfuzero} and standard statements about differentiability of inverse functions.
We have that $X_t(x)=g(\mu_t,t,x)$, where \mbox{$g\in \cc^{1,1,1} (\rr\times [0,T]\times\rr)$}.
For the mapping
$$
(\mu,t,x)\to (\mu,t,g(\mu,t,x)),
$$
the matrix of the first derivatives is non-degenerated. Therefore, the inverse mapping is well-defined and smooth.
\end{remark}

\begin{remark} Let us compare our assumptions with those made in other papers. Usually, it is supposed that $u_0$ is measurable and bounded (see, for example, \cite{flandoli10}, \cite{moham15}, \cite{oltud15}, \cite{wei22dini}). We additionally assume that $u_0$ has a continuous derivative, we need this to guarantee that the symmetric integral of $u_0(X_t^{-1}(x))$ be well-defined.

Condition of differentiability of $b$ is standard, boundedness of $\dfrac{\pr b}{\pr x}$ may be assumed in some $\llf_p$ norm (see \cite{flandoli10}, \cite{wei22dini}) or uniformly (\cite{oltud15}). Note that in~\cite{moham15} the main result was obtained for arbitrary bounded measurable~$b$.

Our integrability condition~A\ref{assfbl} is technical and is important for our method. It is similar to respective assumptions in~\cite{beck19}, \cite{catoli}, \cite{mucha10}.
\end{remark}

\section{Existence of the solution}\label{scexis}

In this section, we prove the first statement of our theorem.

By the chain rule~\eqref{eqintf}, for $\varphi\in\cc_0^{\infty}(\rr)$ we have
\begin{equation}
\begin{split}
{d}_t\Bigl[X_t'(x)\varphi(X_t(x))\Bigr]=\varphi(X_t(x)){d}_t\Bigl[X_t'(x)\Bigr]+X_t'(x) {d}_t\Bigl[\varphi(X_t(x))\Bigr]\\
\stackrel{\eqref{eqaddx},\eqref{eqdifx}}{=}\varphi(X_t(x))  \dfrac{\partial b(t,X_{t}(x))}{\partial x} X'_{t}(x)\,{d}t
+X_t'(x)\varphi'(X_t(x)) b(t,X_{t}(x))\,{d}t +X_t'(x)\varphi'(X_t(x))\circ{d}\mu(t).\label{equchr}
\end{split}
\end{equation}

Applying the change of variables $y=X_t(x)$, we get
\begin{eqnarray*}
\int_{\rr}u_0(X_t^{-1}(y))\varphi(y)\,{d}y=\int_{\rr}u_0(x)X_t'(x)\varphi(X_t(x))\,{d}x\\
=\int_{\rr}u_0(x)\Bigl[X_t'(x)\varphi(X_t(x))\Bigr|_{t=0}+\int_{0}^t {d}_s\bigl[X_s'(x)\varphi(X_s(x))\bigr]\Bigr]\,{d}x\\
\stackrel{\eqref{equchr}}{=}\int_{\rr}u_0(x)\varphi(x)\,{d}x+\int_{\rr}u_0(x)\int_{0}^t\varphi(X_s(x))\dfrac{\partial b(s,X_{s}(x))}{\partial x} X'_{s}(x) \,{d}s\,{d}x\\
+\int_{\rr}u_0(x) \int_{0}^t X_s'(x)\varphi'(X_s(x)) b(s,X_{s}(x))\,{d}s\,{d}x
+\int_{\rr}u_0(x) \int_{0}^t X_s'(x)\varphi'(X_s(x)) \circ{d}\mu_s\,{d}x \\
\stackrel{\eqref{eqlmfb}}{=}\int_{\rr}u_0(x)\varphi(x)\,{d}x+\int_{0}^t \int_{\rr}u_0(x) \varphi(X_s(x))\dfrac{\partial b(s,X_{s}(x))}{\partial x} X'_{s}(x)\,{d}x\,{d}s \\
+\int_{0}^t \int_{\rr}u_0(x) X_s'(x)\varphi'(X_s(x)) b(s,X_{s}(x))\,{d}x\,{d}s
+\int_{0}^t \int_{\rr}u_0(x) X_s'(x)\varphi'(X_s(x)) \,{d}x\,\circ{d}\mu_s .
\end{eqnarray*}
Lemma~\ref{lmlmfb} may be applied here because $\varphi$ has a compact support. Assumption A\ref{assfb} and~\eqref{eqdifx} imply that $C_1\le X_s'\le C_2$ for some positive constants $C_1$ and $C_2$, therefore set $\{x:\ \varphi'(X_s(x))\ne 0\}$ is bounded.

Taking the inverse change of variable $x=X_t^{-1}(y)$, obtain
\begin{eqnarray*}
\int_{\rr}u_0(X_t^{-1}(y))\varphi(y)\,{d}y
=\int_{\rr}u_0(x)\varphi(x)\,{d}x+\int_{0}^t  \int_{\rr}u_0(X_s^{-1}(y)) \varphi(y)\dfrac{\partial b(s,y)}{\partial x} \,{d}y\,{d}s\\
+\int_{0}^t \int_{\rr}u_0(X_s^{-1}(y)) \varphi'(y) b(s,y)\,{d}y \,{d}s+\int_{0}^t  \int_{\rr}u_0(X_s^{-1}(y)) \varphi'(y) \,{d}y\,\circ{d}\mu_s.
\end{eqnarray*}
Thus, $u(t,x)=u_0(X_t^{-1}(x))$ satisfies~\eqref{eqweaks}.

\section{Uniqueness of the solution}\label{scuniq}

In this section, we prove the second statement of our theorem. We will follow the standard approach (see, for example, proof of the uniqueness of the solution in~\cite{catoli}, \cite{mucha10}).

Let $u(t,x)$ satisfies~\eqref{eqweaks} with $u_0(x)=0$. We will obtain that $u(t,x)=0$ what implies the uniqueness of the solution.

For this case, from \eqref{eqweaks} for $\varphi\in\cc_0^{\infty}(\rr)$ we get
\begin{equation}\label{equtxv}
\begin{split}
\int_{\rr}u(t,x)\varphi(x)\,{d}x=\int_0^t  \int_{\rr}u(s,x)\Bigl(b(s,x)\varphi'(x)+\dfrac{\partial b(s,x)}{\partial x}\varphi(x)\Bigr)\,{d}x\,{d}s\\
+\int_0^t \int_{\rr}u(s,x)\varphi'(x)\,{d}x\,\circ {d}\mu_s.
\end{split}
\end{equation}

Denote
$$
G(t,y)=\int_{\rr}u(t,x)\varphi(x-y)\,{d}x
$$
We have that $G(0,y)=0$ because $u(0,x)=0$, and
\begin{eqnarray*}
\dfrac{\pr G(t,y)}{\pr y}=-\int_{\rr}u(t,x)\varphi'(x-y)\,{d}x,
\end{eqnarray*}
Our solution has a form $u(t,x)=h(\mu_t,t,x)$, and, applying \eqref{eqintf} and~\eqref{equtxv}, we obtain
\begin{eqnarray*}
\int_{\rr}u(t,x)\varphi(x-\mu_t)\,{d}x=G(t,\mu_t)\\
=\int_0^t \int_{\rr}b(s,x)u(s,x) \varphi'(x-\mu_s) \,{d}x\,{d}s
+\int_0^t  \int_{\rr}\dfrac{\partial b(s,x)}{\partial x}u(s,x)\varphi(x-\mu_s)\,{d}x\,{d}s\\
+\int_{0}^t \int_{\rr}u(s,x)\varphi'(x-\mu_s)\,{d}x\,\circ{d}\mu_s-\int_{0}^t \int_{\rr}u(s,x)\varphi'(x-\mu_s)\,{d}x\,\circ{d}\mu_s\\
{=}\int_0^t \int_{\rr}b(s,x)u(s,x) \varphi'(x-\mu_s)\,{d}x\,{d}s
+\int_0^t \int_{\rr}\dfrac{\partial b(s,x)}{\partial x}u(s,x)\varphi(x-\mu_s)\,{d}x\,{d}s .
\end{eqnarray*}

For $V(t,z)=u(t,z+\mu_t)$, taking change of the variable $x= z+\mu_t$, get
\begin{equation}\label{eqvtzb}
\begin{split}
\int_{\rr}V(t,z)\varphi(z)\,{d}z=\int_0^t \int_{\rr}b(s,z+\mu_s)V(s,z)\dfrac{d\varphi(z)}{dz}\,{d}z\,{d}s \\
+\int_0^t \int_{\rr}\dfrac{\partial b(s,z+\mu_s)}{\partial z}V(s,z)\varphi(z)\,{d}z\,{d}s .
\end{split}
\end{equation}

Let   $\phi_{\varepsilon}$ be a standard mollifier,
\begin{eqnarray*}
\phi_{\varepsilon}(x)=\dfrac{1}{\varepsilon}\phi\Bigl(\dfrac{x}{\varepsilon}\Bigr),\quad \phi\in \cc_0^{\infty}(\rr),\quad {\rm supp}\,\phi\subset [-1,1],\quad
\phi(x)\ge 0,\quad \int_{\rr}\phi(x)\,{d}x=1.
\end{eqnarray*}
Denote $V_{\varepsilon}(t,x):=V(t,\cdot)*\phi_{\varepsilon}$. Substituting $\varphi(z)=\phi_{\varepsilon}(x-z)$ in~\eqref{eqvtzb}, obtain
\begin{eqnarray*}
V_{\varepsilon}(t,x)=\int_{\rr}V(t,z)\phi_{\varepsilon}(x-z)\,{d}z\\
=-\int_0^t \int_{\rr}b(s,z+\mu_s)V(s,z)\phi'_{\varepsilon}(x-z)\,{d}z\,{d}s
+\int_0^t \int_{\rr}\dfrac{\partial b(s,z+\mu_s)}{\partial z}V(s,z)\phi_{\varepsilon}(x-z)\,{d}z\,{d}s .
\end{eqnarray*}
We take the derivative with respect to $t$, use the notation $B(t,z)=b(t,z+\mu_t)$, and get
\begin{eqnarray*}
\dfrac{\partial V_{\varepsilon}(t,x)}{\partial t}= -\int_{\rr}B(t,z)V(t,z)\dfrac{\partial \phi_{\varepsilon}(x-z)}{\partial x}\,{d}z
+\int_{\rr}\dfrac{\partial B(t,z)}{\partial z}V(t,z)\phi_{\varepsilon}(x-z)\,{d}z\\
=-\dfrac{\partial }{\partial x} \int_{\rr}B(t,z)V(t,z)\phi_{\varepsilon}(x-z)\,{d}z
+\int_{\rr}\Bigl[\dfrac{\partial }{\partial z}[B(t,z)V(t,z)]-B(t,z)\dfrac{\partial V(t,z)}{\partial z}\Bigr]\phi_{\varepsilon}(x-z)\,{d}z\\
\stackrel{(*)}{=}-\dfrac{\partial }{\partial x} ( BV(t,\cdot)*\phi_{\varepsilon})(x)+\dfrac{\partial }{\partial x} ( BV(t,\cdot)*\phi_{\varepsilon})(x)\\
-\int_{\rr}B(t,z)\dfrac{\partial V(t,z)}{\partial z}\phi_{\varepsilon}(x-z)\,{d}z
=-\int_{\rr}B(t,z)\dfrac{\partial V(t,z)}{\partial z}\phi_{\varepsilon}(x-z)\,{d}z.
\end{eqnarray*}

In (*) we have used that $\phi_{\varepsilon}$ has a compact support, and, by integration in parts,
\begin{eqnarray*}
\int_{\rr}\dfrac{\partial }{\partial z}[B(t,z)V(t,z)]\phi_{\varepsilon}(x-z)\,{d}z=
\int_{\rr}\phi_{\varepsilon}(x-z)\,{d}[B(t,z)V(t,z)]\\
=-\int_{\rr}[B(t,z)V(t,z)]\,{d}_z\phi_{\varepsilon}(x-z)=\int_{\rr}[B(t,z)V(t,z)]\dfrac{\partial \phi_{\varepsilon}(x-z)}{\partial x}\,{d}z
=\dfrac{\partial }{\partial x} ( BV(t,\cdot)*\phi_{\varepsilon})(x).
\end{eqnarray*}

Thus,
\begin{eqnarray}\label{eqeqdif}
 \dfrac{\partial V_{\varepsilon}(t,x)}{\partial t}
+\Bigl(B(t,z)\dfrac{\partial V(t,z)}{\partial z}\Bigr)*\phi_{\varepsilon}(x)=0.
\end{eqnarray}

Denote
\begin{equation}\label{eqdifr}
\mathcal{R}_{\varepsilon}(B,V)=\dfrac{\partial V_{\varepsilon}(t,x)}{\partial t}+B(t,x)
\dfrac{\partial V_{\varepsilon}(t,x)}{\partial x}
\stackrel{\eqref{eqeqdif}}{=}B\dfrac{\partial (\phi_{\varepsilon}*V)}{\partial x}-\phi_{\varepsilon}*\Bigl(B\dfrac{\partial V}{\partial x}\Bigr).
\end{equation}

Lemma II.1~i)~\cite{diplio89} gives that for each fixed $t$
\begin{equation}\label{eqleml}
\mathcal{R}_{\varepsilon}(B,V_{\varepsilon})\to 0,\quad \varepsilon\to 0 \quad{\textrm{in}}\quad \llf^1_{loc}(\rr,{d}x)
\end{equation}
provided that $B(t,\cdot)\in {\sf W}^{1,1}_{loc}(\rr)$, $V(t,\cdot)\in \llf^{\infty}_{loc}(\rr,{d}x)$, where ${\sf W}$ denotes the Sobolev space. These conditions hold due to assumptions of our theorem.

Consider $\pi_r(x)=\pi_1(x/r)$, where
$$
\pi_1(x)=
\left\{
\begin{array}{l}
1,\quad |x|<1,\\
\in [0,1],\quad 1\le |x|\le 2,\\
0,\quad |x|>2.
\end{array}
\right.
$$
with $|\pi_r'|\le \dfrac{C}{r}$. We have that
\begin{eqnarray*}
\int_{\rr} {d}_x(B(V_{\varepsilon})^2\pi_r)=0\quad
\Leftrightarrow\quad\int_{\rr} (V_{\varepsilon})^2\pi_r\,{d}_xB+\int_{\rr} B\pi_r \,{d}_x(V_{\varepsilon})^2+\int_{\rr} B(V_{\varepsilon})^2\,{d}\pi_r=0\\
\Leftrightarrow \int_{\rr} B\pi_r V_{\varepsilon} \dfrac{\partial V_{\varepsilon}}{\partial x}\,{d}x=-\dfrac{1}{2}\int_{\rr} (V_{\varepsilon})^2\pi_r\dfrac{\partial B}{\partial x}\,{d}x-\dfrac{1}{2}\int_{\rr} B(V_{\varepsilon})^2 \pi'_r\,{d}x.
\end{eqnarray*}
We multiply \eqref{eqdifr} by $V_{\varepsilon}(t,x)\pi_r(x)$, take the integral over~$\rr$, and get
\begin{eqnarray*}
\int_{\rr} \mathcal{R}_{\varepsilon}(B,V_{\varepsilon}) V_{\varepsilon}\pi_r\,{d}x=\dfrac{1}{2}\int_{\rr}  \dfrac{\partial V^2_{\varepsilon}}{\partial t} \pi_r\,{d}x+ \int_{\rr} B V_{\varepsilon}\pi_r\dfrac{\partial V_{\varepsilon}}{\partial x}\,{d}x\\
\Leftrightarrow \int_{\rr} \mathcal{R}_{\varepsilon}(B,V_{\varepsilon}) V_{\varepsilon}\pi_r \,{d}x=\dfrac{1}{2} \dfrac{\partial }{\partial t} \int_{\rr} V^2_{\varepsilon}\pi_r\,{d}x-\dfrac{1}{2}\int_{\rr} (V_{\varepsilon})^2\pi_r\dfrac{\partial B}{\partial x}\,{d}x
-\dfrac{1}{2}\int_{\rr} B(V_{\varepsilon})^2 \pi'_r\,{d}x.
\end{eqnarray*}

From~\eqref{eqleml} it follows that for fixed $r$ and $t$
$$
\int_{\rr} \mathcal{R}_{\varepsilon}(B,V_{\varepsilon}) V_{\varepsilon}\pi_r \,{d}x\to 0,\quad \varepsilon\to 0.
$$
Therefore,
\begin{equation}
\begin{split}
\lim_{\varepsilon\to 0} \Bigl(\dfrac{\partial }{\partial t} \int_{\rr} V^2_{\varepsilon} \pi_r\,{d}x-\int_{\rr} (V_{\varepsilon})^2\pi_r\dfrac{\partial B}{\partial x}\,{d}x -\int_{\rr} B(V_{\varepsilon})^2 \pi'_r\,{d}x\Bigr)=0\\
\Leftrightarrow
 \dfrac{\partial }{\partial t} \int_{\rr}V^2 \pi_r\,{d}x-\int_{\rr} V^2\pi_r\dfrac{\partial B}{\partial x}\,{d}x =\int_{\rr} BV^2 \pi'_r\,{d}x. \label{eqintt}
\end{split}
\end{equation}
Because $\pi'_r(x)=0$ for $|x|\le r$ or $|x|\ge 2r$, and $|\pi_r'|\le \dfrac{C}{r}$, we have
\begin{equation}\label{eqbvtz}
\begin{split}
\Bigl|\int_{\rr} BV^2 \pi'_r\,{d}x\Bigr|\le \|V\|_{\llf_\infty}^2\int_{r\le |x|\le 2r} \dfrac{|B(t,x)|}{1+|x|}(1+|x|)|\pi'_r(x)|\,{d}x\to 0,\quad
r\to\infty,
\end{split}
\end{equation}
where convergence holds uniformly in~$t$ for each fixed~$\omega$.

In \eqref{eqbvtz} we have used the following estimates. If $\sup_t|\mu_t|=M(\omega)$ then
\begin{eqnarray*}
\int_{r\le |x|\le 2r} \dfrac{|B(t,x)|}{1+|x|}\,{d}x\stackrel{y=x+\mu_t}{=} \int_{r\le |y-\mu_t|\le 2r} \dfrac{|b(t,y)|}{1+|y-\mu_t|}\,{d}y\\
\le \int_{|y|\ge r-M(\omega)} \dfrac{|b(t,y)|}{1+|y|-M(\omega)}\,{d}y\stackrel{A\ref{assfbl}}{\to} 0,\quad r\to\infty.
\end{eqnarray*}

Integrating \eqref{eqintt} in~$t$ and taking into account that
$$
\int_{\rr}V^2(t,x) \pi_r(x)\,{d}x\Bigr|_{t=0}=\int_{\rr}u(0,x)^2\pi_r(x)\,{d}x=0,
$$
we get
\begin{eqnarray}\label{eqvtri}
 \int_{\rr}V^2 \pi_r\,{d}x=\int_0^t \int_{\rr} V^2\pi_r\dfrac{\partial B}{\partial x}\,{d}x\,{d}s +\int_0^t \int_{\rr} BV^2 \pi'_r\,{d}x\,{d}s
\end{eqnarray}

Consider
$$
g_r(t,x)=V^2(t,x) \pi_{r}(x).
$$
By A\ref{assfb}, we have $\Bigl|\dfrac{\partial B}{\partial x}\Bigr|\le K$ for some constant~$K$. From~\eqref{eqvtri}, we get
\begin{eqnarray*}
\int_{\rr} g_r(t,x)\,{d}x\le \int_0^t \int_{\rr} g_r(s,x)\Bigl|\dfrac{\partial B}{\partial x}\Bigr|\,{d}x\,{d}s+R_r
\le K\int_0^t \int_{\rr} g_r(s,x)\,{d}x\,{d}s+R_r,\\
R_r=\sup_t \Bigl|\int_0^t \int_{\rr} BV^2 \pi'_r\,{d}x\,{d}s \Bigr|\stackrel{\eqref{eqbvtz}}{\to}0,\quad r\to \infty.
\end{eqnarray*}

From the Gronwall inequality for $h(t)=\int_{\rr} g_r(t,x)\,{d}x$, we get
$$
\int_{\rr} g_r(t,x)\,{d}x\le R_re^{Kt}.
$$
Taking $r\to\infty$, we get
$$
\int_{\rr} g_r(t,x)\,{d}x\to \int_{\rr} V^2\,{d}x,\quad R_re^{Kt}\to 0\Rightarrow V=0
$$
that finishes the proof of uniqueness of the solution.

\bibliographystyle{bib/vmsta-mathphys}
\bibliography{RadchenkoTranspEq22}

\end{document}